\documentclass[11pt]{amsart}

\advance\hoffset by -0.5 truecm
\usepackage{amssymb}
\usepackage{amsmath}
\usepackage{graphicx}
\newtheorem{Theorem}{Theorem}[section]
\newtheorem{Lemma}[Theorem]{Lemma}
\newtheorem{Corollary}[Theorem]{Corollary}

\newtheorem{Proposition}[Theorem]{Proposition}

\newtheorem{Example}[Theorem]{Example}
\newtheorem{Remark}[Theorem]{Remark}

\def\V{\mbox{Var}}

\def\R\re
\def\V{\bf V}

\def \re{{\mathbb R}}

\def \0{\lambda_{0}}

\DeclareMathOperator{\spn}{span}

\begin{document}
\title[Second Yamabe constant]{Second Yamabe Constant on Riemannian Products}

\author[G. Henry]{Guillermo Henry}\thanks{The author was supported by the SFB 1085 'Higher Invariants' at the  Universit\"at Regensburg, funded by Deutsche Forschungsgemeinschaft  (DFG)}
 \address{Departamento de Matem\'atica, FCEyN, Universidad de Buenos
Aires, Ciudad Universitaria, Pab. I., C1428EHA,
          Buenos Aires, Argentina.}
     
\email{ghenry@dm.uba.ar}


\date{}

\maketitle

\begin{abstract} Let $(M^m,g)$ be a closed Riemannian manifold  $(m\geq 2)$ of positive scalar curvature and $(N^n,h)$ any closed manifold. We study the asymptotic behaviour of the second Yamabe constant and the second $N-$Yamabe constant of $(M\times N,g+th)$ as $t$ goes to $+\infty$. We obtain that  $\lim_{t \to +\infty}Y^2(M\times N,[g+th])=2^{\frac{2}{m+n}}Y(M\times \re^n, [g+g_e]).$   If $n\geq 2$, we show the existence of nodal solutions of the Yamabe equation on $(M\times N,g+th)$ (provided  $t$  large enough). When  $s_g$ is constant, we prove that   $\lim_{t \to +\infty}Y^2_N(M\times N,g+th)=2^{\frac{2}{m+n}}Y_{\re^n}(M\times \re^n, g+g_e)$.   Also we study the second Yamabe invariant and the second  $N-$Yamabe invariant. 
\end{abstract}

\vspace{0.5 cm}

\small
\noindent \text{Keywords:} Second Yamabe constant; Yamabe equation; Nodal solutions.

\vspace{0.5 cm}

\section{Introduction}

$\ $

Let $(W^k,G)$ be  a closed Riemannian manifold of dimension  $k\geq 3$ with scalar curvature $s_G$.  The Yamabe functional  
$J:C^{\infty}(W)-\{0\}\longrightarrow \re$ is defined by

$$ J(u):= \frac{\int_W a_k|\nabla u|^2_G+s_Gu^2dv_G}{\|u\|^2_{p_{k}} }.$$
where  $a_k:=4(k-1)/(k-2)$ and  $p_k:=2k/(k-2)$.

The   infimum of the Yamabe functional over the set of smooth functions of $W$, excluding the zero function, is a conformal invariant and  it is called the Yamabe constant of $W$ in the conformal class of G (which we are going to denote by $[G]$):

$$Y(W,[G])=\inf_{u\in C^{\infty}(W)-\{0\}} J(u).$$
  
  Recall that  the conformal Laplacian operator of $(W,G)$ is

$$L_G:=a_k\Delta_G+s_G,$$  

where $\Delta_G$ is the negative Laplacian, i.e., $\Delta_{g_e}u=-\sum_{i=1}^n\frac{\partial^2 u}{\partial x_i^2}$ in the Euclidean space $(\re^n,g_e)$.

The celebrated Yamabe problem states  that in any conformal class of a closed Riemannian manifold (of dimension at least $3$) there exists a Riemmannian metric with constant scalar curvature. 
This was proved in a series of articles by Yamabe \cite{Y}, Trudinger \cite{Trudinger},  Aubin \cite{Aubin}, and Schoen \cite{Schoen1}.  Actually, they proved that 
the   Yamabe constant is attained by a smooth positive function $u_{min}$. It can be seen that a function $u_{cp}$ is a critical point of the Yamabe functional if and only if  it solves the so called Yamabe equation

\begin{equation}\label{YamabeEquation}L_G(u_{cp})=\lambda |u_{cp}|^{p_k-2}u_{cp}
\end{equation}
 
for $\lambda=J(u_{cp})/\|u_{cp}\|_{p_k}^{p_k-2}$. 
Recall that if $\tilde{G}$ belongs to $[G]$, then   $$L_G(u)=s_{\tilde{G}}u^{p_k-1}$$ where $u$ is  the positive smooth function  that satisfies $\tilde{G}=u^{p_k-2}G$.
Therefore,  $G_{u_{min}}:=u_{min}^{p_k-2}G$ must be  a metric of constant scalar curvature.

 The solution of the Yamabe problem provides a positive smooth solution of the Yamabe equation. Actually, as we pointed  out, 
there is a one to one relationship between the Riemannian metrics with constant  scalar curvature in  $[G]$ and  positive solutions of the Yamabe equation.

Nevertheless, in order to understand the set of solutions of the Yamabe equation, it  seems  important to study the nodal  solutions, i.e., a sign changing solution  of  (\ref{YamabeEquation}). In the last years  several authors have addressed  the question about the existence and multiplicity of nodal solutions of the Yamabe equation: Hebey and Vaugon \cite{Hebey-Vaugon}, Holcman \cite{Holcman}, Jourdain \cite{Jourdain}, Djadli and Jourdain \cite{Djadli-Jourdain}, Ammann and Humbert \cite{Ammann-Humbert}, Petean \cite{Petean2},  El Sayed \cite{ElSayed} among others.

Let 

$$\lambda_1(L_G)<\lambda_2(L_G)\leq \lambda_3(L_G)\leq\dots \nearrow +\infty$$
be the sequence of eigenvalues of $L_G$, where each eigenvalue appears repeated accor\-ding to its multiplicity. It is well known that it is an increasing sequence that tends to infinity.

When $Y(W, [G])\geq 0$, it is not difficult to see  that 

$$Y(W,[G])=\inf_{\tilde{G} \in [G]} \lambda_1(L_{\tilde{G}})vol(W,\tilde{G})^{\frac{2}{k}},$$
 where $vol(W,\tilde{G})$ is the volume of $(W,\tilde{G})$. 
   
In \cite{Ammann-Humbert}, Ammann and Humbert introduced the $l$th Yamabe constant. This cons\-tant is  defined by 

$$Y^l(W,[G]):=\inf_{\tilde{G}\in [G]}\lambda_l(L_{\tilde{G}})vol(W,\tilde{G})^{\frac{2}{k}}.$$ 

Like the Yamabe constant, the $l$th Yamabe constant is a conformal invariant. 

They showed that
  the second Yamabe constant of a connected Riemannian ma\-ni\-fold with non-negative Yamabe constant  is never achieved by a Riemannian metric. Nevertheless, if we enlarge the conformal class, allowing generalized metrics (i.e., metrics of the form $u^{p_k-2}G$ with  $u\in L^{p_k}(W)$,  $u\geq 0$, and $u$ does not vanish identically), under some assumptions on $(W,G)$,  the second Yamabe constant  is achieved (\cite{Ammann-Humbert}, Corollary 1.7).
    Moreover,  if $Y^2(W,G)>0$, they proved  that if a    generalized metric $\tilde{G}$ realizes the second Yamabe constant, then it is    
  of the form $|w|^{p_k-2}G$ with $w\in C^{3,\alpha}(W)$ a nodal solution of the Yamabe equation. If $Y^2(W,G)=0$ and  is attained, then any eigenfunction corresponding to the second eigenvalue of $L_G$ is a nodal solution.

Therefore, if we know that  the second Yamabe constant is achieved, we have a nodal solution of the Yamabe equation. However, this is not the general situation. There  exist some Riemannian manifolds for which    the second Yamabe constant is not achieved,  even by a generalized metric. For instance, $(S^k, g_0^k)$ where $g_0^k$ is the  round metric of curvature $1$ (cf. \cite{Ammann-Humbert}, Proposition 5.3).

Let $(M,g)$ and $(N,h)$ be closed Riemannian manifolds and consider the Riema\-nnian product $(M\times N,g+h)$. We define the $N$-Yamabe constant as
the infimum of the Yamabe functional over the set of   smooth functions, excluding the zero function, that depend only on $N$:

$$Y_N(M\times N,g+h):=\inf_{u\in C^{\infty}(N)-\{0\}} J(u).$$

Clearly, $Y(M\times N,g+h)\leq Y_N(M\times N,g+h)$. The $N-$Yamabe constant is not a conformal invariant, but it is scale invariant. It  was  first introduced  by Akutagawa, Florit, and Petean  in \cite{A-F-P}, where they studied, among other things, its behavior on Riemannian products of the form $(M\times N, g+th)$ with $t>0$.

Actually, the infimum of $J$ over $C^{\infty}(N)-\{0\}$ is a minimum, and it is  achieved by a positive smooth function. 
 
 When the scalar curvature of the product is constant,   the critical points of the Yamabe functional restricted to  $C^{\infty}(N)-\{0\}$, satisfy the Yamabe equation, and thereby, also satisfy the subcritical Yamabe equation   (recall that $p_{m+n}<p_n$). Hence, if $Y_N(M\times N,g+h)=J(u)$, then  the  metric $G=u^{p_{m+n}-2}(g+h)\in[g+h]$ has    constant scalar curvature.
  When  $s_{g+h}\leq 0$, the Yamabe constant of  $(M\times N,g+h)$ is nonpositive, and in this situation,   there is essentially only one metric of constant scalar curvature, the metric $g+h$. Therefore,  this case    it is not interesting.

It seems important to consider the $N-$Yamabe constant because in some cases the minimizer (or some minimizers) of the Yamabe functional depends only on one of the variables of the product. For instance, it was proved  by Kobayashi in \cite{Kobayashi} and Schoen in \cite{Schoen} that the minimizer of the Yamabe functional on $(S^n\times S^1, g_0^n+tg_0^1)$ depends only on $S^1$.  Also, this might be the case for $(S^n\times \mathbb{H}^m, g^n_0+tg_h)$ (for small values of $t$), where $(\mathbb{H}^m, g_h)$ is the $m-$dimensional Hyperbolic space of curvature $-1$. These Riemannian products are interesting, because their Yamabe constants  appear in the surgery formula for the Yamabe invariant (see the definition below) proved by Ammann, Dahl, and Humbert in \cite{A-D-H}.

We define the $l$th $N-$Yamabe constant  as:
 
$$Y^l_{N}(M\times N, g+h):=\inf_{G\in [g+h]_N}\lambda_l^N(L_{G})vol(M\times N ,G)^{\frac{2}{m+n}},$$
where   $[g+h]_N$ is the set of Riemmanian metrics in the conformal class $[g+h]$ that can be written as $u^{p_{m+n}-2}(g+h)$, with $u$ a positive  smooth function that depends only on $N$, and  $\lambda_l^N(L_G)$ is the $l$th eigenvalue of  $L_G$ restricted to functions that depend only on  the variable $N$.  

 A generalized metric  $G=u^{p_{m+n}-2}(g+h)$  is  called a generalized $N-$metric if $u$ depends only  on $N$.

Petean proved (\cite{Petean2}, Theorem 1.1) that the second $N-$Yamabe  constant  of a Riemannian product of closed manifolds with constant and positive scalar curvature is always  attained  by a generalized  $N-$metric of the form $|w|^{p_{m+n}-2}(g+h)$ where $w\in C^{3,\alpha}(N)$  is  a nodal solution of the Yamabe equation. 

The aim of the present article  is study the behaviour of the second Yamabe constant and the second $N-$Yamabe constant of a Riemannian product $(M\times N, g+th)$ with $t>0$. 
 We prove the following results: 

\begin{Theorem} \label{limitY2} Let $(M^m,g)$ be a closed manifold $(m\geq 2)$ with positive scalar curvature and let $(N^n,h)$ be a  closed manifold. Then, 

$$\lim_{t \to +\infty}Y^2(M\times N,[g+th])=2^{\frac{2}{m+n}}Y(M\times \re^n, [g+g_e]).$$

\end{Theorem}

From this theorem, as well as from some results in  \cite{A-F-P} and \cite{Ammann-Humbert}, we obtain:

\begin{Corollary}\label{nodalsolution} Let $(M^m,g)$ as above  and let $(N^n,h)$ be a closed  Riemannian manifold $(n\geq 2)$. For $t$ large enough,  $Y^2(M\times N,[g+th])$ is attained by a generalized metric of the form $|v|^{p_{m+n}-2}(g+th)$, where $v$ is a nodal solution  of  the Yamabe equation on $(M\times N, g+th)$. Moreover, 
$v\in C^{3,\alpha}(M\times N)$ and is smooth in $M\times N-\{  v^{-1}(0)\}$. 
\end{Corollary}

We point out that the nodal solutions provided by  Corollary \ref{nodalsolution}, in general, are not the same solutions provided by (\cite{Petean2}, Theorem 1.1), which depend only on  $N$ (see Subsection \ref{2yamabeconstant} and Remark \ref{YR>Y}).

For the second $N-$Yamabe  constant we obtain the next theorem:

\begin{Theorem}\label{limitY2N} Let $(M^m,g)$ be a closed manifold $(m\geq 2)$ of positive and constant scalar curvature and $(N^n,h)$ be any closed manifold. Then, 

$$\lim_{t \to +\infty}Y^2_N(M\times N,g+th)=2^{\frac{2}{m+n}}Y_{\re^n}(M\times \re^n, g+g_e).$$

\end{Theorem}

In Subsection \ref{lYnoncompact} we will define the second Yamabe constant and the $N-$second Yamabe constant for a non-compact manifold. There we prove:

\begin{Theorem}\label{non compact} Let $(M^m,g)$ be a closed manifold of positive scalar curvature. Then, 

$$Y^2(M\times\re^n,g+g_e)=2^{\frac{2}{m+n}}Y(M\times \re^n, [g+g_e]).$$  

If in addition $(M^m,g)$ has constant  scalar curvature, then 

$$Y^2_{\re^n}(M\times\re^n,g+g_e)=2^{\frac{2}{m+n}}Y_{\re^n}(M\times \re^n, g+g_e).$$ 

\end{Theorem} 

The Yamabe invariant of $W$, which we denote by $Y(W)$,  is the supremum of the Yamabe constants over the set  $\mathcal{M}_W$ of Riemannian metrics on $W$:

$$Y(W):=\sup_{G\in  \mathcal{M}_W}Y(W,[G]).$$

 This  important differential invariant was introduced by Kobayashi in \cite{Kobayashi}  and Schoen in \cite{Schoen1}. It provides information about the capability of $W$ to admit a Rie\-ma\-nnian metric of positive scalar curvature.  More precisely,  the Yamabe invariant is positive   if and only if the manifold admits a metric of positive scalar curvature.

Similarly,   we define  the  $l$th Yamabe invariant of $W$ by

$$Y^l(W):=\sup_{G\in  \mathcal{M}_W}Y^l(W,[G]).$$

For a product   $M\times N$, we define the $l$th $N-$Yamabe invariant as

$$Y^l_N(M\times N):=\sup_{g\in  \mathcal{M}_M^Y,\ h\in\mathcal{{M}}_N}Y_N(M\times N, g+h),$$

where $\mathcal{M}_M^Y$ is the subset of Yamabe metrics of $\mathcal{M}_M$, i.e., metrics that realize the Yamabe constant.
By a result due to Pollack \cite{Pollack} we know that  for any Riemannian manifold of dimension $n\geq 3$ with positive Yamabe invariant there exist metrics with a  constant scalar  curvature $n(n-1)$  and  arbitrarily large volume. Therefore, if we take  the supremum  among $\mathcal{M}_M$ instead of $\mathcal{M}_M^Y$, $Y^l_N(M\times N)$ would be infinite (see the variational characterization of the  $l$th $N$-Yamabe constant in Section \ref{Preliminaries}).

In Section \ref{section2inv}, we point out several facts about the second Yamabe invariant and the second $N-$Yamabe invariant.  Also, taking into account  some known bounds for the Yamabe invariant,  we show lower bounds for these invariants.

Note that frequently  in the literature, the Yamabe constant and the Yamabe invariant are  called Yamabe invariant and $\sigma-$invariant, respectively. Something similar happens for the $l$th Yamabe invariant and for  the $l$th Yamabe constant. In this article we are not going to use these denominations.

\subsection*{Acknowledgements}
The author would like to thank the hospitality
of the members of SFB 1085 Higher Invariant at the University of  Regensburg, where he stayed
during the preparation of this work. He would like to express his gratitude to Bernd Ammann
for very helpful discussions, remarks, and for sharing his expertise.  Also, he would like to thank Bernd Ammann's research group for their kind hospitality. He would like to thank to Jimmy Petean for  many valuable  conversations and  useful observations. Finally, the author would like to thank to the anonymous reviewer for his/her valuable comments and suggestions on the manuscript.

\section{Preliminaries}\label{Preliminaries}

\subsection{Notation.}

$\ $

Let $(W^k,G)$ be a Riemannian manifold. Throughout this article we will denote with $C^{\infty}_{\geq 0}(W)$   and  $L^{p}_{\geq 0}(W)$  the set of   non-negative functions on $W$,  excluding the  zero function, that belong to $C^{\infty}(W)$ and   $L^p(W)$, respectively. We are going to denote with 
  $C^{\infty}_{>0}(W)$  the  positive functions of $C^{\infty}_{\geq 0}(W)$.  $L^{p}_{\geq 0,\ c}(W)$  and $C^{\infty}_{\geq 0,\ c}(W)$  will be  the subset of functions with   compact support that belong to $L^{p}_{\geq 0}(W)$ and   $C^{\infty}_{\geq 0}(W)$, respectively.

Let $H$ be one of these spaces of functions:  $C^{\infty}(W)$, $C^{\infty}_c(W)$ or $H^2_1(W)$. We  write  $Gr^l(H)$  for the set of all $l-$dimensional subspaces of $H$.  If $u\in H$, we denote with $Gr^l_u(H)$ the elements of $Gr^l(H)$ that satisfy: If  $V=\spn(v_1,\dots,v_l)$, then   $\tilde{V}=\spn(u^{p_k-2}v_1,\dots, u^{p_k-2}v_l)$   belongs to  $Gr^l(H)$.

\subsection{Results from the literature}\label{resultsliterature}

$\ $

Here, for the convenience of the reader, we state some important results from the literature that we are going to use in the next sections.

 The following theorem is due to Ammann and Humbert (\cite{Ammann-Humbert}, Theorem 5.4 and Proposition 5.6):

\begin{Theorem}\label{AH-bounds}
Let  $(W^k,G)$ be a closed Riemannian manifold $(k\geq 3)$ with non-negative Yamabe constant. Then,    

\begin{equation}\label{bound2yamabeconstant}
 \nonumber2^ {\frac{2}{k}}Y(W,[G])\leq Y^2(W,[G])\leq
[Y(W,[G])^ {\frac{k}{2}}+Y(S^{k})^{\frac{k}{2}}]^{\frac{2}{k}}.
\end{equation}

$\ $

Moreover,  if $Y^2(W,[G])$ is attained and $W$ is connected, then the left hand side inequality is strict.

\end{Theorem} 

We summarise the main results   of \cite{Ammann-Humbert}  (Theorem 1.4, 1.5, and 1.6) in the next theorem:

\begin{Theorem}\label{mainresultAH} Assume  the same hypothesis as in the theorem above:

\begin{itemize}
\item[a)] $Y^2(W,[G])$ is attained by a generalized metric if

$$Y^2(W,[G])<[Y(W,[G])^ {\frac{k}{2}}+Y(S^{k})^{\frac{k}{2}}]^{\frac{2}{k}}.$$

Furthermore, if $Y^2(W,[G])>0$ this generalized metric is of the form $|w|^{p-2}G$ with $w\in C^{3,\alpha}(W)$  a nodal solution of the Yamabe equation.

\item[b)] The inequality in $a)$ is fulfilled  by  any non locally conformally flat manifold with 
$Y(W,[G])>0$ and $k\geq 11$ or  
 $Y(W,[G])=0$ and $k\geq 9$. 
\end{itemize}

\end{Theorem}

In \cite{A-F-P}, Akutagawa, Florit, and Petean studied the behavior of the Yamabe cons\-tant and the $N-$Yamabe constant  on Riemannian products. More precisely,  they proved the following important result (\cite{A-F-P}, Theorem 1.1):

\begin{Theorem}\label{thmAFP} Let $(M^m,g)$ and $(N^n,h)$ be closed Riemannian manifolds. In addi\-tion, assume that $(M,g)$ is  of positive scalar curvature and $m\geq 2$. Then,  

$$\lim_{t \to +\infty}Y(M\times N,[g+th])=Y(M\times\re^n,[g+g_e]),$$

and 

$$\lim_{t \to +\infty}Y_N(M\times N,g+th)=Y_{\re^n}(M\times\re^n,g+g_e).$$

\end{Theorem}

If $(M,g)$ is a closed manifold, then $(M\times\re^n, g+g_e)$ is  complete, with positive injective radius  and  bounded geometry. Hence,  the Sobolev embedding theorem holds (cf. \cite{Hebey}, Theorem 3.2).     If we assume that  the scalar curvature $s_g$ is positive, then it is not difficult to see that $Y(M^m\times \re^n,g+g_e)>0$ (see  Section \ref{YCnoncompact} for the definition of the Yamabe constant in the non-compact case).  If $m,n\geq 2$,  
  it was proved   in  (\cite{A-F-P}, Theorem 1.3) that 

\begin{equation}\label{upperboundYMR}
0<Y(M\times\re^n,[g+g_e])<Y(S^{m+n}).
\end{equation}

\subsection{Yamabe constant  on   non-compact manifolds}\label{YCnoncompact}

$\ $

Note that in the definition of the  Yamabe constant  
the infimum of the Yamabe functional  could be taken as well over $C^{\infty}_{>0}(W)$, $C^{\infty}_c(W)-\{0\}$ or $H^2_1(W)-\{0\}$ and  it does not change. 
Thus, it seems natural (cf. \cite{SchoenYau})  to define the Yamabe constant 
of  a non-compact manifold $(W^k,G)$   as

$$Y(W,[G]):=\inf_{u\in C^{\infty}_c(W)-\{0\} } \frac{\int_W a_k|\nabla u|^2_G+s_Gu^2dv_G}{\|u\|^2_{p_{k}} }.$$

The Yamabe constant, also in the noncomapct setting,  is always  bounded from above by the Yamabe constant of $(S^n,g^n_0)$. Since $Y(S^k,[g_0^k])=Y(S^k)$, we have that  $Y(W)\leq Y(S^k)$.    

\subsection{Variational characterization of the $l$th Yamabe constant}\label{variationalcharact}

$\ $

 It is well known the min-max characterization of the $l$th eigenvalue of conformal Laplacian of a closed manifold $(W^k,G)$:

\begin{eqnarray}
\nonumber\lambda_l(L_G) &  = &\inf_{V\in Gr^l(C^{\infty}(W))}\sup_{v\in V-\{0\}}\frac{\int_W L_G(v)vdv_G}{\|v\|_2^2}  \\
\nonumber& & \\  
 \nonumber &=&\inf_{V\in Gr^l(H^2_1(W))}\sup_{v\in V-\{0\}}\frac{\int_W a_k|\nabla v|^2_G+s_Gv^2dv_G}{\|v\|_2^2}.  
\end{eqnarray}

For any Riemannian metric  $G_u:=u^{p_k-2}G$  in $[G]$, the conformal Laplacian sa\-tis\-fies  the invariance property 

$$L_{G_u}(v)=u^{1-p_k}L_G(uv).$$ 

Since  $vol(W,G_u)=\int_W u^{p_k}dv_G$,  we get  

$$\lambda_l(L_{G_u})vol(W,G_u)^{\frac{2}{k}}=
\inf_{V\in Gr^l(H^2_1(W))}\sup_{v\in V-\{0\}}\frac{\int_W a_k|\nabla v|^2_G+s_Gv^2dv_G}{\int_Wu^{p_{k}-2}v^2dv_G}$$$$\times(\int_Wu^{p_k}dv_G)^{\frac{2}{k}}.$$

Therefore, we have the following characterization of the  $l$th Yamabe constant of $(W,G)$: 

 $$Y^l(W,[G])=\inf_{\substack{u\in C^{\infty}_{>0}(W)\\ V\in Gr^l(H^2_1(W))} }\sup_{v\in V-\{0\}}\frac{\int_W a_k|\nabla v|^2_G+s_Gv^2dv_G}{\int_Wu^{p_{k}-2}v^2dv_G}(\int_Wu^{p_k}dv_G)^{\frac{2}{k}}.$$

If we enlarge the conformal class of $G$, allowing generalized metrics, then we obtain

$$Y^l(W,[G])=\inf_{\substack{
            u\in L^{p_{k}}_{\geq 0}(W)\\
            V\in Gr_u^l(H^2_1(W))}}\sup_{v\in V-\{0\}}\frac{\int_W a_k|\nabla v|^2_G+s_Gv^2dv_G}{\int_Wu^{p_{k}-2}v^2dv_G}(\int_Wu^{p_k}dv_G)^{\frac{2}{k}}.$$

 Let $(M^m\times N^n,g+h)$ be a Riemannian product of closed manifolds with $s_g$  constant. If we consider generalized $N-$metrics instead of $N-$metrics in $[g+h]$, we have  the following variational characterization of the $l$th $N-$Yamabe constant:

$$Y^l_N(M\times N,g+h)=\inf_{\substack{
            u\in L^{p_{m+n}}_{\geq 0}(N)\\
            V\in Gr_u^l(H^2_1(N))}}\sup_{v\in V-\{0\}}\frac{\int_N a_k|\nabla v|^2_{g+h}+s_{g+h}v^2dv_{g+h}}{\int_Nu^{p_{m+n}-2}v^2dv_{g+h}}$$$$\times\big(\int_Nu^{p_{m+n}}dv_{g+h}\big)^{\frac{2}{m+n}}\big(vol(M,g)\big)^{\frac{2}{m+n}}.$$

\section{Second Yamabe constant and second $N-$Yamabe constant on Riemannian products}

\subsection{Second Yamabe constant}\label{2yamabeconstant}

$\ $

Let $(M^m, g)$  be  a closed manifold ($m\geq 2$) of positive scalar curvature,  and   $(N^n,h)$   any closed   Riemannian manifold. Note that $Y(M\times N,[g+th])$ is positive  for $t$ large enough. By Theorem \ref{AH-bounds}, we get 

$$2^ {\frac{2}{k}}Y(M\times N,[g+th]) \leq  Y^2(M\times N,[g+th])\leq [Y(M\times N,[g+th])^{\frac{k}{2}}+Y(S^{k})^{\frac{k}{2}}]^{\frac{2}{k}}, $$

where $k=m+n$. Applying Theorem \ref{thmAFP}  to these inequalities, we obtain the following lemma:

\begin{Lemma}\label{limitc}
Let $(M^m, g)$  be a closed manifold $(m\geq 2)$  of positive scalar curvature  and let  $(N^n,h)$ be  any closed manifold.   Then, 

$$2^ {\frac{2}{m+n}}Y(M\times \re^ n, [g+g_e])\leq \liminf_{t \to +\infty}Y^2(M\times N,[g+th])$$

and

$$\limsup_{t \to +\infty}Y^2(M\times N,[g+th])\leq [Y(M\times\re^n,[g+g_e])^ {\frac{m+n}{2}}+Y(S^{m+n})^{\frac{m+n}{2}}]^{\frac{2}{m+n}}.$$
 \end{Lemma}

When $(M,g)$ is $(S^{m-1},g^{m-1}_0)$ with $m\geq 3$ and $(N,h)$ is $(S^1,g^1_0)$ the  lemma  above implies that 

$$
\lim_{t \to +\infty}Y^2(S^{m-1}\times S^1,g^{m-1}_0+tg^1_0)=2^{\frac{2}{m}}Y(S^{m}).
$$

Here, we used that $Y(S^{m-1}\times \re,g^{m-1}_{0}+g_e)=Y(S^m)$. But, by the inequality (\ref{upperboundYMR}), this is no longer true for $(S^{m-1}\times \re^n,g^{m-1}_{0}+g_e)$ when $n\geq 2$.

\begin{proof}[Proof of Theorem \ref{limitY2}] From  Lemma \ref{limitc}  we only have to prove that 

$$\limsup_{t \to +\infty}Y^2(M\times N,[g+th])\leq2^{\frac{2}{m+n}}Y(M\times \re^n, [g+g_e]).$$ 

 Given $\varepsilon>0$,  let $f=f_{\varepsilon}\in C^{\infty}_{\geq 0,\ c}(M\times \re^n)$   such that 

$$J(f)\leq Y(M\times \re^n,[g+g_e])+\varepsilon.$$ 

 Assume that the support of $f$ is included in $ M\times B_R(0)$,   where $B_R(0)$ is the Euclidean ball centred at $0$ with radius $R$.

  For $q\in N$, we denote with $\exp_{q}^h$  the exponential map at $q$ with respect to the metric $h$  and with  $B_{\delta}^h(0_{q})$  the ball of radius $\delta$  centred at $0_{q}\in T_{q}N$.

   Let $q_1$ and $q_2$ be two points on $N$,  and consider their normal neighbourhoods  
$U_1=\exp_{q_1}^h(B_{\delta}^h(0_{q_1}))$ and $U_2=\exp_{q_2}^h(B_{\delta}^h(0_{q_2}))$.
We are going to choose $\delta>0$, such that $U_1$ and $U_2$ are disjoint sets and for  any  normal coordinate system $x=(x_1,\dots,x_n)$, we have

$$(1+\varepsilon)^{-1}dv_{g_e}\leq dv_h\leq(1+\varepsilon)dv_{g_e}.$$

 Note that for  the metric $t^2h$, we have $B_{\delta}^h(0_{q_i})=B_{t\delta}^{t^2h}(0_{q_i})$. Therefore, if we consider a normal coordinate system $y=(y_1,\dots,y_n)$ 
 with respect to the metric $t^2h$, we get 
 $$(1+\varepsilon)^{-1}dv_{g_e}\leq dv_{t^2h}\leq(1+\varepsilon)dv_{g_{e}}$$
 in $B_{t\delta}^{t^2h}(0_{q_i})$.

Let $t_1$ be such that $t_1\delta>R$. For $t\geq t_1$,  we are going to identify $B_{t\delta}(0)\subseteq R^n$ with $U_i=\exp_{q_i}^{t^2h}(B_{t\delta}^{t^2h}(0_{q_i}))$. Hence, 

$$M\times B_R(0)\subseteq M\times B_{t\delta}(0)\simeq M\times U_i\subseteq M\times N.$$

Let $\phi_i:M\times N\longrightarrow \re$  
 defined by 
 $$\phi_i(p,q):= \left\{
\begin{array}{c l}
 f(p,q) & (p,q) \in M\times U_i,\\
 0 &  (p,q) \not\in M\times  U_i,\\
 \end{array}
\right.$$

and  let us consider $\phi:M\times N\longrightarrow \re$ given by $$\phi:=\phi_1+\phi_2.$$

 Clearly, $\phi_i\in C^{\infty}_{\geq 0}(M\times N)$, $\phi\in L^{p_{m+n}}_{\geq 0}(M\times N)$, and 
 the subspace  $V_0:=\spn(\phi_1,\phi_2)$ belongs to  $Gr_{\phi}^2(M\times N)$.   

If we choose $t_2$ such that  $s_{g+th}\leq (1+\epsilon)s_g$ for ${t}\geq {t_2}$, then taking ${t}\geq t_3:= \max(t_1^2,{t_2})$,  it is not difficult to see that   
\begin{equation}\label{boundmetric1}
\int_{M\times U_i}a_{m+n}|\nabla \phi_i|_{g+th}^ 2+s_{g+th}\phi_i^2 dv_{g+th}
\end{equation}

$$
 \leq (1+\varepsilon)^3 \int_{M\times B_{R}(0)}a_{m+n}|\nabla \phi_i|_{g+g_e}^ 2+s_{g}\phi_i^2 dv_{g+g_e},
$$  

and 

\begin{equation}\label{boundmetric2}
\int_{M\times B_R(0)} \phi_i^{p_{m+n}}dv_{g+g_e}\leq  (1+\varepsilon)\int_{M\times U_i} \phi_i^{p_{m+n}}dv_{g+th}.
\end{equation}

By the variational characterization of the second Yamabe constant we get   

$\ $ 

 $Y^2(M\times N,[g+th])$

\begin{align*}
 &\leq \sup_{v\in V_0-\{0\}}\frac{\int_{M\times N}a_{m+n}|\nabla v|_{g+th}^ 2+s_{g+th}v^2 dv_{g+th}}{\int_{M\times N}\phi^{p_{m+n}-2}v^2dv_{g+th}}\\
&\ \ \ \ \ \ \ \ \ \ \ \ \ \ \ \ \ \ \ \ \ \ \ \times\Big(\int_{M\times N} \phi^{p_{m+n}}dv_{g+th}\Big)^{\frac{2}{m+n}}\\
&=\sup_{(\alpha_1,\alpha_2)\in \re^2-\{0\}}\frac{\sum_{i=1}^2\alpha_i^2(\int_{M\times N}a_{m+n}|\nabla \phi_i|_{g+th}^ 2+s_{g+th}\phi_i^2 dv_{g+th})}{\int_{M\times N}\alpha^2_1\phi_1^{p_{m+n}}+\alpha_2^2\phi_2^{p_{m+n}}dv_{g+th}}\\
&\ \ \ \ \ \ \ \ \ \ \ \ \ \ \ \ \ \ \ \ \ \ \  \times \Big(\int_{M\times N} \phi_1^{p_{m+n}}+\phi_2^{p_{m+n}}dv_{g+th}\Big)^{\frac{2}{m+n}}\\
&=  2^{\frac{2}{m+n}}\sup_{(\alpha_1,\alpha_2)\in \re^2-\{0\}}\frac{\sum_{i=1}^2\alpha_i^2(\int_{M\times N}a_{m+n}|\nabla \phi_i|_{g+th}^ 2+s_{g+th}\phi_i^2 dv_{g+th})}{(\alpha^2_1+\alpha_2^2)\|\phi_1\|_{p_{m+n}}^{2}}
\end{align*}

In the last equality, we used that  $\|\phi_1\|_{p_{m+n}} =\|\phi_2\|_{p_{m+n}}$. Applying the inequality (\ref{boundmetric2}), we obtain    

$$Y^2(M\times N,[g+th])\leq 2^{\frac{2}{m+n}}(1+\varepsilon)^{\frac{2}{p_{m+n}}} $$

$$\times\sup_{(\alpha_1,\alpha_2)\in \re^2-\{0\}}\frac{\sum_{i=1}^2\alpha_i^2(\int_{M\times N}a_{m+n}|\nabla \phi_i|_{g+th}^ 2+s_{g+th}\phi_i^2 dv_{g+th})}{(\alpha^2_1+\alpha_2^2) (\int_{M\times B_R(0)} \phi_1^{p_{m+n}}dv_{g+g_e})^{\frac{2}{p_{m+n}}}}.$$

By inequality (\ref{boundmetric1}), for any  $t\geq t_3$, we have

$$
Y^2(M\times N,[g+th])\leq  (1+\varepsilon)^{\frac{4(m+n)-2}{m+n}}2^{\frac{2}{m+n}}$$

$$\times\frac{\int_{M\times B_R(0)}a_{m+n}|\nabla f|_{g+g_e}^ 2+s_{g}f^2 dv_{g+g_e}}{(\int_{M\times B_R(0)}f^{p_{m+n}}dv_{g+g_e})^{\frac{2}{p_{m+n}}}}
= (1+\varepsilon)^{\frac{4(m+n)-2}{m+n}}2^{\frac{2}{m+n}}J(f)$$

$$
\leq  (1+\varepsilon)^{\frac{4(m+n)-2}{m+n}}2^{\frac{2}{m+n}}\big(Y(M\times \re^n,g+g_e)+\varepsilon\big)
.$$

Finally, letting $\varepsilon$ go to $0$,  we obtain that

  $$\limsup_{t\to \infty} Y^2(M\times N,[g+th])\leq 2^{\frac{2}{m+n}} Y(M\times \re^n,[g+g_e]),
  $$
which finishes the proof.

\end{proof}

\begin{Remark} The same proof can be adapted to prove that  

$$\limsup_{t \to +\infty}Y^l(M\times N,[g+th])\leq l^{\frac{2}{m+n}}Y(M\times \re^n,[g+g_e]),$$
for  $l\geq 2$.

\end{Remark}

\begin{Corollary}\label{boundforY2attained} Let $(M^m,g)$ be a closed manifold $(m\geq 2)$ with positive scalar curvature and let $(N^n,h)$  be any closed manifold $(n\geq 2)$. Then, for $t$ large enough,  we have

$$Y^2(M\times N,[g+th])<[Y(M\times N,[g+th])^ {\frac{m+n}{2}}+Y(S^{m+n})^{\frac{m+n}{2}}]^{\frac{2}{m+n}}.$$

\end{Corollary}

\begin{proof}  Since  $Y(M\times\re^n,[g+g_e])<Y(S^{m+n})$, it follows that

 $$2^{\frac{2}{m+n}}Y(M\times \re^n, [g+g_e])<[Y(M\times \re^n ,[g+g_e])^ {\frac{m+n}{2}}+Y(S^{m+n})^{\frac{m+n}{2}}]^{\frac{2}{m+n}}.$$

	On the other hand, we know by Theorem \ref{thmAFP} that   $\lim_{t\to+\infty}Y(M\times N,[g+th])=Y(M\times \re^n ,[g+g_e])$.  Thereby,  provided  $t$ large enough,   Theorem \ref{limitY2} implies the desired inequality.

\end{proof}

	Now, Corollary  \ref{nodalsolution} is an immediate  consequence of the corollary  above and  Theorem \ref{mainresultAH}. Hence, for $t$ large enough, we have a   sign changing solution $v\in C^{3,\alpha}(M\times N)$ of the equation 
	
	 $$L_{g+th}v=\lambda |v|^{p_{m+n}-2}v.$$  
	 We can choose $v$ such that $\lambda=Y^2(M\times N,[g+th])$.
  
		  Note that in general $(M\times N,g+th)$ is not locally conformally flat (eventually it is when $(M,g)$ and $(N,th)$ have constant sectional curvature $1$ and $-1$).  Therefore, when $m+n\geq 11$,   Corollary \ref{nodalsolution}  is a direct consequence of  Theorem \ref{mainresultAH}. 
		  
	Actually, as we mentioned in the Introduction,  the second $N$-Yamabe constant of a product $(M\times N,g+th)$ is attained (when  $s_g$ or $s_{g+h}$ is constant) by  a generalized $N-$metric, and this provides  a nodal solution of the Yamabe equation on $(M\times N,g+th)$ that only depends on $N$, i.e., a nodal solution of $$L_{g+h}(w)=Y^2_N(M\times N,g+h) |w|^{p_{m+n}-2}w.$$
 However,  in general, this solution is  not  the same solution that the one  provided by Corollary \ref{nodalsolution}. The reason is that   $Y^2(M\times N,[g+th])$,  generally, will be smaller than $Y^2_N(M\times N,g+th)$ (see  Remark \ref{YR>Y}).


\subsection{Second $N-$Yamabe constant.}\label{subsectionN-Yamabe}
$\ $

The second $N-$Yamabe constant  is always attained by a generalized metric.
It can be  proved, with the same argument used in \cite{Petean2},  that the $l$th  $N-$Yamabe constant is also attained by a generalized metric.
$\ $

\begin{Lemma}\label{lowerbound2Nyamabe} Let $(M,g)$  and $(N,h)$ be  closed Riemannian manifolds such that $s_g$ is constant and  $Y_N(M\times N,g+h)\geq 0$. Then,  

$$2^ {\frac{2}{m+n}}Y_N(M\times N,g+h)\leq Y^2_N(M\times N,g+h).$$

\end{Lemma}

The argument  to prove the lemma  is similar to the one used to prove the first inequality in Theorem \ref{AH-bounds} (for the details see the proof of Proposition 5.6 in  \cite{Ammann-Humbert}). In this situation we only have to restrict to functions that depend only on the $N$ variable.  For convenience of the reader we briefly sketch the proof:

\begin{proof}
For $u\in L^{p_{m+n}}(N)$ and $v\in H^2_1(N)-\{0\}$, let us consider $$F_N(u,v)=\frac{\big(\int_{N} a_{m+n}|\nabla v|^2_{h}+s_{g+h}v^2dv_{h}\big)\big(\int_N u^{p_{m+n}}dv_h\big)^{\frac{2}{m+n}}vol(M,g)^{\frac{2}{m+n}}}{\int_{N}u^{p_{m+n}-2}v^2dv_{h}}.$$
The lemma will follow if we prove that for any $u\in C^{\infty}_{>0}(N)$, with $\|u\|_{p_{m+n}}=1$, and any $V\in Gr^2(C^{\infty}(N))$ we have
\begin{equation}\label{supremum}
 \sup_{v\in V-\{0\}}F_N(u,v)\geq 2^ {\frac{2}{m+n}}Y_N(M\times N,g+h).
\end{equation}

The operator $L_{u^{p_{m+n}-2}(g+h)}$ restricted to $H^2_1(N)$ has a discrete spectrum $$0<\lambda_1^N(L_{u^{p_{m+n}-2}(g+h)})\leq \lambda_2^N(L_{u^{p_{m+n}-2}(g+h)})\leq\dots$$

Let $w_1$ and $w_2$ be the first two eigenvectors associated with  $\lambda_1^N(L_{u^{p_{m+n}-2}(g+h)})$ and $\lambda_2^N(L_{u^{p_{m+n}-2}(g+h)})$, respectively. By the conformal invariance of the conformal Laplacian operator, $v_1=uw_1$ and $v_2=uw_2$ satisfy 

$$
 L_{g+h}(v_1)=\lambda_1^N(L_{u^{p_{m+n}-2}(g+h)})u^{p_{m+n}-2}v_1$$
and
$$
 L_{g+h}(v_2)=\lambda_2^N(L_{u^{p_{m+n}-2}(g+h)})u^{p_{m+n}-2}v_2.
$$

We can choose $w_1$ and $w_2$ such that    
$$
 \int_{N}u^{p_{m+n}}v_1v_2dv_{h}=0.$$

 Notice that by the maximum principle we can also choose $v_1>0$, then $v_2$ must change sign. 

The supreme (\ref{supremum}) in any subspace $V\in Gr^2(C^{\infty}(N))$ is  greater or equal  than $\sup_{v\in V_0-\{0\}}F_N(u,v)$ when  $V_0:=span(v_1,v_2)$. Actually, we have that  
$$\sup_{v\in V_0-\{0\}}F_N(u,v)=\lambda_2^N(L_{u^{p_{m+n}-2}(g+h)}).$$

Now, using the H\"{o}lder inequality and the definition of the $N-$Yamabe constant we get

$$2Y_N(M\times N,g+h)\leq \lambda_2^N(L_{u^{p_{m+n}}(g+h)})\Big[\Big(\int_{\{v_2\geq0\}} u^{p_{m+n}-2}dv_{h} \Big)^{\frac{p_{m+n}-2}{p_{m+n}}}$$
   $$+\Big(\int_{\{v_2<0\}} u^{p_{m+n}}dv_{h}\Big)^{\frac{p_{m+n}-2}{p_{m+n}}} \Big].$$

Applying again  the H\"{o}lder inequality, we obtain 
   
   $$\Big(\int_{\{v_2\geq 0\}} u^{p_{m+n}}dv_{h} \Big)^{\frac{p_{m+n}-2}{p_{m+n}}}
   +\Big(\int_{\{v_2< 0\}}u^{p_{m+n}}dv_{h}\Big)^{\frac{p_{m+n}-2}{p_{m+n}}} \leq 2^{\frac{2}{p_{m+n}}}.$$
   Therefore, $$2^{\frac{2}{m+n}}Y_N(M\times N,g+h)\leq\lambda_2^N(L_{u^{p_{m+n}-2}(g+h)}).$$


\end{proof}

\begin{proof}[Proof of Theorem \ref{limitY2N}]

 By the positiveness of the scalar curvature of  $(M,g)$, for any $t>0$ we have 
 
 $$0<Y(M\times N,[g+th])\leq Y_N(M\times N,g+th).$$
 
 Hence, by Lemma \ref{lowerbound2Nyamabe}  we have 

$$2^ {\frac{2}{m+n}}Y_N(M\times N,g+th)\leq Y^2_N(M\times N,g+th).$$   

From Theorem  \ref{thmAFP},   we obtain 

   $$ 2^ {\frac{2}{m+n}}Y_{\re^n}(M\times\re^n,g+g_e)\leq \liminf_{t \to +\infty}  Y^2_{N}(M\times N, g+th).$$

For  any $\varepsilon>0$,  we choose  
$f=f_{\varepsilon}\in C^{\infty}_{\geq 0, c}(\re^{n})$   that satisfies  

$$J(f) \leq Y_{\re^n}(M\times \re^n,g+g_e)+\varepsilon,$$

 then, it can be proved by  a similar argument  to the one used in the proof of Theorem \ref{limitY2} that
 $$\limsup_{t \to +\infty}  Y^2_{N}(M\times N, g+th)\leq 2^ {\frac{2}{m+n}}Y_{\re^n}(M\times\re^n,g+g_e).$$ 
This completes the proof.

\end{proof}

\begin{Remark} Let $(M^m,g)$ be a closed Riemannian manifold  $(m\geq 2)$ of constant  positive scalar curvature. Hence by Theorem \ref{thmAFP} we have that  $Y(M\times \re^n,[g+g_e])=Y_{\re^n}(M\times \re^n,g+g_e)$ if and only if 
$$\lim_{t \to +\infty}Y_N(M\times N,g+th)=\lim_{t \to +\infty}Y(M\times N,[g+th]).$$
for  any closed Riemannian manifold $(N,h)$.  By Theorem   \ref{limitY2} and Theorem \ref{limitY2N}, the equality $Y(M\times \re^n,[g+g_e])=Y_{\re^n}(M\times \re^n,g+g_e)$ is also equivalent to have 
$$\lim_{t \to +\infty}Y^2_N(M\times N,g+th)=\lim_{t \to +\infty}Y^2(M\times N,[g+th]),$$ 
for any closed   Riemannian manifold $(N,h)$.

\end{Remark}

\vspace{0.5 cm}

For $m$ and $n$  positive integers, the   $\alpha_{m,n}$  Gagliardo-Nirenberg constant is defined as

 $$\alpha_{m,n}:=\Big[\inf_{u\in H^2_1(\re^n)-\{0\} } \frac{(\int_{\re^n} |\nabla u|^2dv_{g_e})^{\frac{n}{m+n}}(\int_{\re^n}u^2dv_{g_e})^{\frac{m}{m+n}}}{(\int_{\re^n}|u|^{p_{m+n}}dv_{g_e})^{\frac{m+n-2}{m+n}} }\Big]^{-1}.$$

These constants are positive and can be computed numerically. In  \cite{A-F-P}, they were computed for some cases ($m+n\leq 9$, with $n,m \geq 2$).  Also it was proved in (\cite{A-F-P}, Theorem 1.4) that  for any  closed Riemmannian manifold $(M,g)$ of positive constant  scalar curvature  and with unit volume, it holds 
 
 \begin{equation}\label{estimationRnyamabe}
 Y_{\re^n}(M\times \re^n,g+g_e)=\frac{A_{m,n}s_g^{\frac{m}{m+n}}}{\alpha_{m,n}},
 \end{equation}
 where $A_{m,n}:=(a_{m+n})^{\frac{n}{m+n}}(m+n)m^{-\frac{m}{m+n}}n^{-\frac{n}{m+n}}$.

An immediate consequence of  (\ref{estimationRnyamabe}) is:

\begin{Corollary}\label{gagliardolimitY2N} Let $(M^m,g)$ be a closed manifold $(m\geq 2)$ of  constant positive  scalar curvature and $(N^n,h)$  any closed Riemannian manifold. Then, 

$$\lim_{t \to +\infty}Y^2_N(M\times N,g+th)=\frac{2^{\frac{2}{m+n}}A_{m,n}s_g^{\frac{m}{m+n}}vol(M,g)^{\frac{2}{m+n}}}{\alpha_{m,n}}.$$

\end{Corollary}

\begin{Remark}\label{YR>Y} If $(W,G_s)=(M^m\times N^n,s^{-n}g+s^{m}h)$ where $(M,g)$ and $(N,h)$ are closed manifolds of constant   positive scalar curvature and unit volume, then $(W,G_s)$ has constant positive scalar curvature and unit volume too. Nevertheless, the scalar curvature of $(W,G_s)$ tends to infinity as $s$ goes to infinity. Therefore,  for $s$ large enough, from  $(\ref{estimationRnyamabe})$  we obtain that  
$ Y(S^{m+n+k})<Y_{\re^k}(W\times \re^k,G_s+g_e)$, hence  $Y(W\times \re^k,[G_s+g_e])< Y_{\re^k}(W\times \re^k,G_s+g_e)$.
This implies that,  for  any closed  $k-$dimensional manifold $(Z,w)$ and $t$  sufficiently large, we have

\begin{align*}
Y(W\times Z, [G_s+tw])&<Y_Z(W\times Z,G_s+tw),
\end{align*}

and 

 $$Y^2(W\times Z,[G_s+tw])<Y^2_Z(W\times Z,G_s+tw).$$

\end{Remark}


\subsection{Second  Yamabe and  second $N-$Yamabe constant on non-compact manifolds.}\label{lYnoncompact}

$\ $

Throughout this section,   $(W^k,G)$ will be a complete Riemannian manifold, not  necessary  compact, with $Y(W,[G])>0$. We define the $l$th Yamabe constant of $(W,G)$ as

$$Y^l(W,G):=\inf_{\substack{
            u\in L^{p_{k}}_{\geq 0, c
            }(W)\\
            V\in Gr_u^l(C^{\infty}_c(W))}}\sup_{v\in V-\{0\}}\frac{\int_W a_k|\nabla v|^2_G+s_Gv^2dv_G}{\int_Wu^{p_{k}-2}v^2dv_G}\big(\int_Wu^{p_k}dv_G\big)^{\frac{2}{k}}.$$

\begin{Proposition} For $l\geq 2$,  $0<Y(W,G)=Y^1(W,G)\leq Y^l(W,G)$.
\end{Proposition}

\begin{proof}  
To prove that $Y(W,G) \leq Y^l(W,G)$ for $l\geq 1$, it is sufficient to show   that  

\begin{equation}\label{YlY1}
Y(W,G)\leq \sup_{v\in V-\{0\}}\frac{\int_W a_k|\nabla v|^2_G+s_Gv^2dv_G}{\int_Wu^{p_{k}-2}v^2dv_G}
\end{equation}

 for any    $u\in L^{p_{k}}_{\geq 0, c}(W)$ with  $\|u\|_{p_k}=1$ and  $V\in Gr_u^l(C^{\infty}_c(W))$. 
 
  If $v\in V-\{0\}$, by the H\"{o}lder inequality, we have that 

$$0<\int_W u^{p_k-2}v^2dv_G\leq (\int_W v^{p_k}dv_G)^{\frac{2}{p_k}}.$$

Since $Y(W,[G])>0$,  we have that $\int_W a_k|\nabla v|^2_G+s_Gv^2dv_G>0$ for any  $v\in V-\{0\}$. Thereby, we obtain   

$$J(v)=\frac{\int_{W}a_k|\nabla v|^2_G+s_Gv^2dv_G}{(\int_{W}v^{p_k}dv_G)^{\frac{2}{p_{k}}}} \leq \frac{\int_{W}a_k|\nabla v|^2_G+s_Gv^2dv_G}{\int_{W}u^{p_{k}-2}v^2dv_G}.$$

Now, taking supreme on the right hand side of the last inequality  we get (\ref{YlY1}).

Let $u_i\in C^{\infty}_{\geq 0,\ c}(W)$  be a minimizing sequence of $Y(W,[G])$. We can assume that $\|u_i\|_{p_k}=1$.  Then,

\begin{align*}
Y(W,[G]) & \leq Y^1(W,G) \leq  \inf_{\substack{V\in Gr_{u_i}^1(C^{\infty}_c(W))\\v\in V-\{0\}}}\frac{\int_W a_k|\nabla v|^2_G+s_Gv^2dv_G}{\int_Wu^{p_{k}-2}_iv^2dv_G}\\
&\leq  \int_W a_k|\nabla u_i|^2_G+s_Gu_i^2dv_G=J(u_i)\underset{i\to +\infty}{\longrightarrow }Y(W,[G]),\\
\end{align*}

which finishes the proof.
\end{proof}

Let $(M^m,g)$ be a closed Riemannian manifold  of constant scalar curvature and let $(N^n,h)$ be a non-compact Riemannian manifold such that   $Y(M\times N,[g+h])>0$. Then, we define the $l$th $N-$Yamabe constant of $(M\times N,g+h)$ as

$$Y^l_N(M\times N,g+h):=\inf_{\substack{
            u\in L^{p_{m+n}}_{\geq 0, c
            }(N)\\
            V\in Gr_u^l(C^{\infty}_c(N))}}\sup_{v\in V-\{0\}}\frac{\int_N a_{m+n}|\nabla v|^2_h+s_{g+h}v^2dv_h}{\int_Nu^{p_{m+n}-2}v^2dv_h}$$
            
            $$\ \ \ \ \ \ \ \ \ \ \ \ \ \ \ \times\big(\int_Nu^{p_{m+n}}dv_G\big)^{\frac{2}{m+n}}\big(vol(M,g)\big)^{\frac{2}{m+n}}.$$

$\ $

\begin{proof} [Proof of Theorem \ref{non compact}]

We are going to prove  the statement  of the theorem  for the second Yamabe constant case.  
 The argument to show the assertion for the second $N-$Yamabe constant is similar. We only have to restrict to functions that depend only on $\re^n$. 
  
  The proof essentially follows along the lines as that of Theorem 4.1 in \cite{Ammann-Humbert}.  
  
  First we are going to show that 

$$Y^2(M\times\re^n,g+g_e)\leq2^{\frac{2}{m+n}}Y(M\times \re^n, [g+g_e]).$$  

Let  $\varepsilon>0$ and consider   $f=f_{\varepsilon}\in C^{\infty}_{\geq 0,\ c}(M\times \re^n)$ such that 

$$J(f)\leq Y(M\times \re^n,[g+g_e])+\varepsilon.$$ 

Assume that the support of $f$ is in $M\times B_R(0)$.  For $\tilde{R}>2R$,  we can choose $q_1$ and $q_2$ in $B_{\tilde{R}}(0)$ such that  $B_R(q_1)\cap B_R(q_2)=\emptyset$  and    $M\times B_R(q_1)\cup M\times B_R(q_2)\subset  M\times B_{\tilde{R}}(0)$. Consider the function $u:=v_1+v_2$  where $v_i(p,q)=f(p,q-q_i)$, and let $V_0:=\spn(v_1,v_2)\in Gr_u^2(C^{\infty}_{ c}(M\times \re^n))$. Then, 

\begin{align*}
Y^2(M\times \re^n,g+g_e)& \leq  \sup_{v\in V_0-\{0\}}\frac{\int_{M\times B_{\tilde{R}}(0)} a_{m+n}|\nabla v|^2_{g+g_e}+s_{g}v^2dv_{g+g_e}}{\int_{M\times  B_{\tilde{R}}(0)}u^{p_{m+n}-2}v^2dv_{g+g_e}}\\& \ \ \ \ \ \ \ \ \ \times\big(\int_{M\times B_{\tilde{R}}(0)}u^{p_{m+n}}dv_{g+g_e}\big)^{\frac{2}{m+n}}\\ & \leq 2^{\frac{2}{m+n}}J(f)\leq 2^{\frac{2}{m+n}}\Big(Y(M\times \re^n,[g+g_e])+\varepsilon\Big).
\end{align*}

Letting  $\varepsilon$ go to $0$, we obtain the desired inequality.

Let  $u\in L^{p_{m+n}}_{\geq 0, c}(M\times \re^n)$ and $V\in Gr_{u}^2(C^{\infty}_c(M\times \re^n))$. Let us consider $F(u,V)$ given by 
$$F(u,V):=\sup_{v\in V-\{0\}}\frac{\int_{M\times \re^n} a_{m+n}|\nabla v|^2_{g+g_e}+s_{g}v^2dv_{g+g_e}}{\int_{M\times \re^n}u^{p_{m+n}-2}v^2dv_{g+g_e}}\big(\int_{M\times \re^n}u^{p_{m+n}}dv_{g+g_e}\big)^{\frac{2}{m+n}}.$$

Since $H: L^{p_{m+n}}_{\geq 0, c}(M\times \re^n) \times C^{\infty}_c(M\times \re^n)-\{0\}\longrightarrow \re$ defined by 
$$H(u,v):=\frac{\int_{M\times \re^n} a_{m+n}|\nabla v|^2_{g+g_e}+s_{g}v^2dv_{g+g_e}}{\int_{M\times \re^n}u^{p_{m+n}-2}v^2dv_{g+g_e}}\big(\int_{M\times \re^n}u^{p_{m+n}}dv_{g+g_e}\big)^{\frac{2}{m+n}}.$$
is continous, then $F$ depends continuously on $u$ and $V$.

Let $u\in C^{\infty}_c(M\times \re^n)$ be a non-negative function with support   included in $M\times B_R(0)$.    We claim that for any $V\in Gr_u^{2}(C^{\infty}_c(M\times B_R(0)))$, 

$$F(u,V)\geq 2^{\frac{2}{m+n}}Y(M\times \re^n,[g+g_e]).$$  
Without loss of generality we can assume that $\|u\|_{p_{m+n}}=1$. 
Let $k$ be a positive integer, we define

$$u_k(p,q):= 
 \frac{u(p,q)+\frac{1}{k}\chi_{(M\times \overline{B_R(0)})}(p,q)}{\|u+\frac{1}{k}\chi_{(M\times \overline{B_R(0)})}\|_{p_{m+n}}} $$

where $\chi_{(M\times \overline{B_R(0)})}$ is the characteristic function of  $M\times \overline{B_R(0)}$.

We are going to proceed  in a similar manner to the proof of Lemma \ref{lowerbound2Nyamabe}.  
Let us consider the operator  $P_k:C^{\infty}_c(M\times B_R(0))\longrightarrow C^{\infty}_c (M\times B_R(0))$ defined by 

$$P_k(v):=a_{m+n}u_k^{\frac{2-p_{m+n}}{2}}\Delta_{g+g_e}(u_k^{\frac{2-p_{m+n}}{2}}v)+s_g u_k^{(2-p_{m+n})}v.$$

 If $\lambda_1^k\leq \lambda_2^k$ are the first two eigenvalues of the Dirichlet problem for $P_k$ on $M\times \overline{B_R(0)}$,  and  $v_1^k$ and $v_2^k$ their respective associated eigenvectors, then 
  $u_k^{-\frac{p_{m+n}}{2}}v_1^k$  and  $u_k^{-\frac{p_{m+n}}{2}}v_2^k$ are eigenvectors of  the  conformal Laplacian  $L_{u_k^{p_{m+n}-2}(g+g_e)}$ with eigenvalues $\lambda^k_1$ and  $\lambda^k_2$, respectively.  We can choose  $v_1^k$ and  $v_2^k$ such that for $w_1:=u_k^{\frac{2-p_{m+n}}{2}}v^k_1$ and $w_2:=u_k^{\frac{2-p_{m+n}}{2}}v^k_2$   we have
 
 \begin{equation}\label{z1}
 L_{g+g_e}(w_1)=\lambda_1^ku_k^{p_{m+n}-2}w_1,\\
 \end{equation}
 
 \begin{equation}\label{z2}
 L_{g+g_e}(w_2)=\lambda_2^ku_k^{p_{m+n}-2}w_2,
\end{equation}
 and  

 \begin{equation}\label{z12}
 \int_{M\times\re^n}u_k^{p_{m+n}-2}w_1w_2dv_{g+g_e}=0.
 \end{equation}
 
 By the maximum principle, $w_1$ has no zeros in $M\times B_R(0)$. Hence, we can assume that $w_1>0$ in $M\times B_R(0)$.   Therefore, by equation (\ref{z12}),   $w_2$ must change sign in $M\times  B_R(0)$. Let, $z_1:=a \max(0,w_2)$ and $z_2:=b\max(0,-w_2)$. We choose   $a$, $b$ $\in \re_{>0}$    such that 
  
  \begin{equation}\label{ll}\int_{M\times \re^n} u_k^{p_{m+n}-2}z_l^2dv_{g+g_e}=1,\end{equation} 
 for $l=1,2$.

Then, by the H\"older inequality,  we have

$$2= \int_{M\times \re^n} u_k^{p_{m+n}-2}z_1^2dv_{g+g_e} +\int_{M\times \re^n} u_k^{p_{m+n}-2}z_2^2dv_{g+g_e}$$

$$\leq  (\int_{\{w_2\geq 0\}} u_k^{p_{m+n}}dv_{g+g_e} )^{\frac{p_{m+n}-2}{p_{m+n}}}(\int_{M\times B_R(0)} z_1^{p_{m+n}}dv_{g+g_e})^{\frac{2}{p_{m+n}}} $$

$$  +(\int_{\{w_2<0\}} u_k^{p_{m+n}}dv_{g+g_e}   )^{\frac{p_{m+n}-2}{p_{m+n}}}(\int_{M\times B_R(0)} z_2^{p_{m+n}}dv_{g+g_e})^{\frac{2}{p_{m+n}}} .$$

By the definition of the Yamabe constant, we obtain

$\ $

 $$2Y(M\times B_R(0),[g+g_e]) 
\leq\Big[\Big(\int_{\{w_2\geq 0\}} u_k^{p_{m+n}}dv_{g+g_e} \Big)^{\frac{p_{m+n}-2}{p_{m+n}}}$$
$$\ \ \ \ \ \ \times\Big(\int_{M\times B_R(0)}L_{g+g_e}(z_1)z_1dv_{g+g_e}\Big)$$
 
 $$+\Big(\int_{\{w_2<0\}} u_k^{p_{m+n}}dv_{g+g_e}\Big)^{\frac{p_{m+n}-2}{p_{m+n}}}\Big(\int_{M\times B_R(0)}L_{g+g_e}(z_2)z_2dv_{g+g_e}\Big)\Big].$$
 
 From equations (\ref{z1}), (\ref{z2}), and (\ref{ll}) we get that
 
   $$2Y(M\times B_R(0),[g+g_e]) \leq \lambda^k_2\Big[\Big(\int_{\{w_2\geq 0\}} u_k^{p_{m+n}}dv_{g+g_e} \Big)^{\frac{p_{m+n}-2}{p_{m+n}}}$$
   
   $$+\Big(\int_{\{w_2<0\}} u_k^{p_{m+n}}dv_{g+g_e}\Big)^{\frac{p_{m+n}-2}{p_{m+n}}} \Big].$$
 
 Then, applying  again the H\"older inequality, we have that 
 
 $$ 2Y(M\times B_R(0),[g+g_e] )\leq \lambda_2^k2^{\frac{2}{p_{m+n}}}.$$
 
 Therefore, $$2^{\frac{2}{m+n}}Y(M\times B_R(0),[g+g_e])\leq\lambda_2^k.$$
 
 Since  $\lambda_2^k=\inf_{V\in Gr_{u_k}^{2}(C^{\infty}_c(M\times B_R(0)))} F(u_k,V)$, we have proved the claim for $u_k$. By the continuity of $F$ with respect to the first variable,   letting $k$ go to infinity we obtain that 
 
 $$F(u,V)\geq  2^{\frac{2}{m+n}}Y(M\times B_R(0),[g+g_e])\geq   2^{\frac{2}{m+n}}Y(M\times \re^n,[g+g_e])$$ for any $V\in Gr_u^{2}(C^{\infty}_c(M\times B_R(0)))$.
 
  Therefore, for any $u\in C^{\infty}_c(M\times \re^n)$ and $V\in  Gr_{u}^{2}(C^{\infty}_c(M\times \re^n))$, we can choose $R$ sufficiently large such that $u\in C^{\infty}_c(M\times B_R(0)) $ and $V\in  Gr_{u}^{2}(C^{\infty}_c(M\times B_R(0)))$, and then we apply the claim.  
  
 Thereby, we obtain that 
  
  $$Y^2(M\times\re^n,g+g_e)\geq 2^{\frac{2}{m+n}}Y(M\times\re^n,[g+g_e]).$$
   
 $\ $

 \end{proof}

 As a consequence of  Theorem \ref{non compact},  we can rewrite the statements of   Theorem \ref{limitY2}  and Theorem \ref{limitY2N} as follows:

 \begin{Theorem}
 Let $(M^m,g)$ be a closed manifold $(m\geq 2)$ with positive scalar curvature and let $(N^n,h)$  be any  closed manifold. Then,   
 
 $$\lim_{t \to +\infty}Y^2(M\times N,[g+th])=Y^2(M\times\re^n,g+g_e).$$
 
If in addition $s_g$ is constant, then

 $$\lim_{t \to +\infty}Y^2_N(M\times N,g+th)=Y_{\re^n}^2(M\times \re^n,g+g_e).$$

\end{Theorem}


\section{Second Yamabe and  second $N-$Yamabe invariant}\label{section2inv}

Throughout this section  $W^k$ will  be a closed manifold of dimension $k$.

 If $Y(W,[G])\geq 0$, then $Y(W,[G])=Y^1(W,[G])$. Therefore, we have  that $Y(W)=Y^1(W)$ if $W$ admits a metric of constant scalar curvature equal to zero. Recall that if  $Y(W)>0$, then $W$ admits such  metrics (cf. \cite{K-W}).
 By (\cite{Ammann-Humbert}, Proposition 8.1), we know that if $Y^l(W,[G])<0$, then $Y^l(W,[G])=-\infty$. 
Hence, if  $Y(W)<0$ or $Y(W)=0$ and the Yamabe invariant is not   attained, then the first Yamabe invariant of $W$ must be  $-\infty$.

Note that  the infimum of the $l$th Yamabe constant   over the space of Riemannian metrics of $W$ is always $-\infty$. Indeed, for every positive integer $l$, we can  find a metric $G$  such that the first $l$ eigenvalues of $L_G$ are negative (cf.  \cite{ElSayed}, Proposition 3.2), which    implies that $Y^1(M,[G])=\cdots=Y^l(M,[G])=-\infty$.

\subsection{Second Yamabe Invariant}

 \begin{Proposition}  $Y^2(W)>-\infty$ if and only if $Y^2(W)\geq 0$.
 \end{Proposition}
 
 \begin{proof} Suppose that $Y^2(W)<0$. Then, the second Yamabe constant of any metric  $G$   is negative, which  implies that   $Y^2(W,[G])=-\infty$. Therefore,  $Y^2(W)=-\infty$. 
 \end{proof}

 \begin{Lemma}\label{samesigneigenvalue}  Let $[G]$ be a conformal class of $W$ and let $\tilde{G}\in[G]$. Then,  $\lambda_l(L_G)$ and $\lambda_l(L_{\tilde{G}})$  have the same sign.
\end{Lemma}

\begin{proof}  Let $u\in C^{\infty}_{>0}(W)$ such that  $\tilde{G}=u^{p-2}G$. Assume  that $\lambda_l(L_G)>0$ and $\lambda_l(L_{\tilde{G}})\leq 0$. Let $V_0\in Gr^l(H^2_1(W))$  that realizes $\lambda_l(L_{\tilde{G}})$. Then,   

$$\sup_{v\in V_0-\{0\}}\frac{\int_WvL_G(v)dv_G}{\int_Wu^{p_k-2}v^2dv_G}=\lambda_l(L_{\tilde{G}})\leq0,$$   

which implies that  $\int_WvL_g(v)dv_G\leq0$ for any   $v\in V_0-\{0\}$. Therefore, we obtain  

$$0<\lambda_l(L_G)\leq \sup_{v\in V_0-\{0\}}\frac{\int_WvL_G(v)dv_G}{\int_Wv^2dv_G}\leq 0,$$
which is a contradiction. Hence,    $\lambda_l(L_{\tilde{G}})> 0$. 

Now, assume that  $\lambda_l(L_G)=0$. Is easy to see that $\lambda_l(L_{\tilde{G}})$ can not be negative. If   $\lambda_l(L_{\tilde{G}})>0$, then we are in the same situation as above.  Exchanging $G$ by $\tilde{G}$, we get that $\lambda_l(L_G)>0$, which is again a contradiction. Thus, $\lambda_l(L_{\tilde{G}})=0$.

\end{proof}

\begin{Lemma}\label{sign2eigenvalue} $Y^2(W)=-\infty$ if and only if the second eigenvalue of the conformal Laplacian is negative for all the Riemannian metrics on $W$.
\end{Lemma}

\begin{proof}
If for any metric  $\lambda_2(L_G)<0$, then $Y^2(W,[G])=-\infty$. Thus, if this is fulfilled for all the metrics on $W$,  then   $Y^2(W)=-\infty$.  

Now assume that   $Y^2(W)=-\infty$. Therefore, for any metric $G$  we have  

$$Y^2(W,G)=\inf_{h\in [G]}\lambda_l(L_h)vol(W,h)^{\frac{2}{k}}=-\infty.$$ 

Hence,  there exists a metric $\tilde{G}$ in the conformal class  $[G]$ with $\lambda_2(L_{\tilde{G}})<0$. By Lemma \ref{samesigneigenvalue},  $\lambda_2(L_G)$ must be negative. 
\end{proof}

\begin{Proposition} If  $Y^2(W)=-\infty$, then  $Y(W)\leq 0$.
 \end{Proposition}
 
 \begin{proof}   Lemma  \ref{sign2eigenvalue} implies that  $\lambda_2(L_G)<0$ for any metric $G$ on $W$. Therefore,  the first eigenvalue of $L_G$ is negative,  and consequently  $Y(W,[G])<0$.  Thereby,   $Y(W)\leq 0$. 
 \end{proof}

 \begin{Example}
 
 $\ $
 
   a) Let $M$ be a closed manifold with $Y(M)<0$. For instance, take  $M={\mathbb{H}^ 3} / \Gamma$ any compact quotient of the 3-dimensional Hyperbolic space.
Let us consider $W:=M\sqcup M$, the disjoint union of two copies of $M$.  We denote with $M_i$ $(i=1,2)$  the  copies of $M$. If $G$ is any metric on $W$, let us denote by $G_i$ the restriction of $G$ to $M_i$. Recall   as the sign of the first eigenvalue of the conformal Laplacian has the same sign that the Yamabe constant.   Thereby, 
 $$\lambda_2(L_G)=\min\Big( \max_{i=1,2}(\lambda_1(L_{G_i})), \lambda_{2}(L_{G_1}),\lambda_{2}(L_{G_2}) \Big)<0$$
and 
$$Y^2(M\sqcup M)=-\infty.$$

 b)    Let $M$ be a compact quotient of a non abelian nilpotent Lie group. It is known that  $Y(M)=0$ but the Yamabe invariant  is not attained by any conformal class.   Then,    $W=M\sqcup M$   satisfies  that $Y^2(W)=-\infty$ and $Y(W)=0$.

 \end{Example}

\begin{Proposition} If $W$ admits a metric of zero scalar curvature, then  $Y^2(W)>0$. 
\end{Proposition}

\begin{proof} If $Y(W)>0$, then it is clear that $Y^2(W)>0$. Assume that $Y(W)=Y(W,[G_0])=0$ for some metric $G_0$.  Then,  $\lambda_1(L_{G_0})=0$ and $\lambda_2(L_{G_0})>0$. Therefore, $Y^2(W,[G_0])\geq 0$. If $Y^2(W,[G_0])>0$, then we have nothing to prove. If $Y^2(W,[G_0])=0$,  then by Theorem \ref{mainresultAH} part a)  the second Yamabe constant  is achieved by a generalized metric $\tilde{G}$. Therefore $\lambda_2(L_{\tilde{G}})=0$, which is a contradiction.  

\end{proof}

\begin{Remark} Let $N$ be a closed manifold obtained  by performing surgery  on $(W,G)$ of codimension at least 3.  B\"ar and Dahl proved in $($\cite{Bar-Dahl}, Theorem 3.1$)$ that  
 given $l\in $ and $\varepsilon>0$ there exists a metric $H$ on $N$ such that $|\lambda_i(L_H)-\lambda_i(L_G)|<\varepsilon$ for  $1\leq i \leq l$. Therefore, the positivity of the second Yamabe invariant is preserved under surgery of codimension at least 3.  
\end{Remark}


 \subsection{Bounds for the second Yamabe invariant and the second $N-$Yamabe invariant.}

 $ \ $

An immediate consequence of the Theorem \ref{AH-bounds} is the following proposition: 
 
 \begin{Proposition}\label{bound2yamabeinvariant} If $W$  admits a metric of zero scalar curvature, then  
 
 \begin{equation}
 \nonumber 2^ {\frac{2}{k}}Y(W)\leq Y^2(W)\leq
[Y(W)^ {\frac{k}{2}}+Y(S^{k})^{\frac{k}{2}}]^{\frac{2}{k}}.
\end{equation}
 
 \end{Proposition}
 
 If $W=S^k$,  $Y^2(S^k)=2^{\frac{2}{k}}Y(S^k)$. From Theorem \ref{AH-bounds}, we have that  $Y^2(S^k,[g^k_0])=2^{\frac{2}{k}}Y(S^k)$. Hence,  the second Yamabe invariant of $S^k$ is achieved by the second Yamabe constant of the conformal class  $[g_0^k]$. But recall that $Y^2(S^k,[g^k_0])$   is not achieved, even by a generalized metric.

 Also, it follows from Proposition \ref{bound2yamabeinvariant} that the second Yamabe invariant of a $k-$dimensional manifold is bounded from above  by the second Yamabe invariant of the $k-$dimensional sphere:  
 
 $$Y^2(W)\leq Y^2(S^k).$$
 
 \begin{Example}  Let $G$ be the Riemannian metric on $S^k\sqcup S^k$ whose restriction to each copy of $S^k$ is $g_0^k$.
  Then,  $Y^2(S^k\sqcup S^k,G)=2^{\frac{2}{k}}Y(S^k)$ $($ cf. \cite{Ammann-Humbert}, Proposition 5.1$)$. Thus, 
$ Y^2(S^k\sqcup S^k)= Y^2(S^k)$.
 \end{Example}
 
 \begin{Example} Let $W=S^{k-1}\times S^1$ $(k\geq 3)$. Using that $Y(S^{k-1}\times S^1)=Y(S^k)$ $($ cf. \cite{Kobayashi} and \cite{Schoen}$)$ it follows from   Proposition \ref{bound2yamabeinvariant} that  
$Y^2(S^{k-1}\times S^1)=2^{\frac{2}{k}}Y(S^k)$.
 \end{Example}

 \begin{Example} It was computed by LeBrun in \cite{LeBrun} that $Y(\mathbb{C}P^2)=12\sqrt{2}\pi$. Then,  $24\pi \leq Y^2(\mathbb{C}P^2)\leq 4\sqrt{42}\pi$.  
 
    Bray and Neves proved  in \cite{Bray-Neves} that $Y(\mathbb{R}P^3)=2^{-\frac{2}{3}}Y(S^3)$. Therefore,  the second Yamabe invariant of    $\mathbb{R}P^3$ is bounded by $Y(S^3)\leq Y^2(\mathbb{R}P^3)\leq (\frac{3}{2})^{\frac{2}{3}}Y(S^3)$. 
    
    Both,  are examples where the second Yamabe invariant is positive but strictly minor than the second Yamabe invariant of the sphere.
 
 \end{Example}

 Let $M^m$ and $N^n$ be closed manifolds $(m,n\geq 2)$ with positive Yamabe invariant.  An immediate consequence of Theorem \ref{limitY2} is that  
 $$Y^2(M\times N)
 \geq 2^ {\frac{2}{m+n}}\sup_{\{ s_g>0,   s_h>0\}} \max\big(Y(M\times \re^n,[g+g_e^n]), Y(N\times \re^m,[h+g_e^m])\big).$$
 For $S^n\times S^n$, we get that $ Y^2(S^n\times S^n)\geq2^{\frac{1}{n}}Y(S^n\times\re^n,[g_0^n+g^n_e])$. Note that if $Y(M)>0$ and $N$ is any closed manifold, then  $Y^2(M\times N)>0$.

  \vspace{0.4 cm}

In the following  proposition we use several known lower bounds  for the Yamabe invariant  to deduce lower bounds for the second Yamabe invariant of a Riemannian product.

 \begin{Proposition}\label{proplowerboundsY2fromYbounds} $\ $
 
 \begin{itemize}

 \item[i)] Let $M^m\times N^n$ with  $m,n\geq 3$ and $Y(M)>0$. Then, 
 
 $$Y^2(M\times N)\geq 2^ {\frac{2}{m+n}}B_{m,n}Y(M)^{\frac{m}{m+n}}Y(S^n)^{\frac{n}{m+n}}.$$
 
 where $B_{m,n}=a_{m+n}(m+n)(ma_{m})^{-\frac{m}{m+n}}(na_n)^{-\frac{n}{m+n}}$.

 \item[ii)]  Let $M$ be a 2-dimensional closed manifold.  Then, 

$$Y^2(M\times S^2)\geq \frac{2c}{3^ {\frac{3}{4}}}Y(S^4),$$
 
 where $c=(1.047)^2$.
 
 \item [iii)]Let  $(M^m,g)$ be a closed manifold with Ricci curvature bounded from below by $(m-1)$. Then, 
 
 $$Y^2(M\times S^1)\geq 2^ {\frac{2}{m+1}}\Big(\frac{vol(M,g)}{vol(S^m,g_0^m)}\Big)^{\frac{2}{m+1}}Y(S^{m+1}).$$
 
 \item [iv)]Let $M^3$ and $N^2$ be closed manifolds.  Then,
 
  $$Y^2(M\times S^2)\geq 2^ {\frac{2}{5}}(0.62)Y(S^5)$$ 
  
  and 
  
  $$Y^2(N\times S^3)\geq 2^ {\frac{2}{5}}(0.75)Y(S^5).$$

 \end{itemize}

 \end{Proposition} 
 
 The statements in Proposition \ref{proplowerboundsY2fromYbounds} are immediate consequence of apply Proposition \ref{bound2yamabeinvariant} to the lower bounds for the Yamabe invariant obtained in \cite{Ammann-Dahl-Humbert}, \cite{Petean},    \cite{Petean-Ruiz},   and \cite{Petean-Ruiz2}.   In all the cases, in order to obtain the bounds,   Theorem \ref{thmAFP}  (first equali\-ty)  
 is used.
 In  \cite{Petean-Ruiz}  and   \cite{Petean}, the authors estimated the isoperimetric profile of $S^2\times \re^2$ and  $M\times S^1$   and used them   to obtain lower bounds for $Y(M\times \re^2)$   and $Y(M\times \re)$ respectively. In \cite{Petean-Ruiz2}, the authors compare the isoperimetric profile of $S^2\times\re^3$ and $S^3\times \re^2$ with the one of $S^5$, and used it to obtain a lower bounds of  $Y(S^2\times \re^3,[g^2_0+g_e])$ and $Y(S^3\times \re^2,[g^3_0+g_e])$.
 In the following,  for convenience of the reader,  we state  the bounds obtained by Ammann, Dahl, and Humbert, Petean, and Petean and Ruiz:
 
 $\ $

 \begin{itemize}
 \item[i)] In \cite{Ammann-Dahl-Humbert},  Ammann, Dahl, and Humbert  proved that the Yamabe invariant of  a Riemannian product $M^m\times N^n$ with  $m,n\geq 3$ and $Y(M)\geq0$ is bounded from below by
 
 $$Y(M\times N)\geq B_{m,n}Y(M)^{\frac{m}{m+n}}Y(S^n)^{\frac{n}{m+n}}.$$

 \item[ii)]  In  \cite{Petean-Ruiz},  Petean and Ruiz proved that for any $2-$dimensional manifold $M$  
 
 $$Y(M\times S^2)\geq \frac{\sqrt{2}c}{3^ {\frac{3}{4}}}Y(S^4).$$

\item[iii)] It was proved by Petean in \cite{Petean} that if $(M^m,g)$ is a closed Riemannian manifold with $Ricci(g)\geq (m-1)g$, then 

$$Y(M\times \re,[g+g_e])\geq \Big(\frac{vol(M,g)}{vol(S^m,g_0^m)}\Big)^{\frac{2}{m+1}}Y(S^{m+1}).$$

\item[iv)] In \cite{Petean-Ruiz2}, Petean and Ruiz  proved that if $M$ is a closed $3-$dimensional mani\-fold and if $N$ is a closed $2-$dimensional ma\-ni\-fold, then $Y(M\times S^2)\geq 0.63Y(S^5)$ and  $Y(N\times S^3)\geq 0.75Y(S^5)$.

 \end{itemize}


\begin{Proposition}\label{Nproductlowerbound} Let $M^m$ be a closed manifold with $Y(M)>0$ and  $N^n$  any  closed manifold. Then, 

$$Y_N^2(M\times N)\geq \frac{2^{\frac{2}{m+n}}A_{m,n}Y(M)^{\frac{m}{m+n}}}{\alpha_{m,n}}.$$

\end{Proposition}

\begin{proof} Let $g$ be a Yamabe metric  with positive Yamabe constant and  unit volume. Let $h$ be any Riemannian metric on $N$.  From Theorem \ref{limitY2N} and Corollary  \ref{gagliardolimitY2N}  we obtain 

$$ Y_N^2(M\times N)\geq \lim_{t\to+\infty}Y^2_N(M\times N,g+th) $$

$$=2^{\frac{2}{m+n}}Y_{\re^n}(M\times \re^n,g+g_e)=\frac{2^{\frac{2}{m+n}}A_{m,n}Y(M,[g])^{\frac{m}{m+n}}}{\alpha_{m,n}}.$$

The proposition follows taking the supreme over the set of Yamabe metrics  on $M$ with unit volume.     
\end{proof}

\begin{Example}   From the proposition above we get that $Y^2_{S^2}(S^2\times S^2)\geq 84.01080$   and $Y^3_{S^3}(S^3\times S^3)\geq 119.33249$. Here, we used the numerical computations of the Glariardo-Nirenberg constants carried out in \cite{A-F-P}, i.e,  $\alpha_{2,2}=0.41343$ and $\alpha_{3,3}=0.31257$.  
\end{Example}


\begin{thebibliography}{aa}

\bibitem{A-F-P} K. Akutagawa, L. Florit, and J. Petean, {\it On Yamabe constant of Riemannian  Products}, Communications in Analysis and Geometry
 {\bf 15} (2007), 947-969.







\bibitem{Ammann-Humbert} B. Ammann and E. Humbert, {\it The second Yamabe invariant},
Journal of Functional Analysis {\bf 235} (2006), 377-412.

\bibitem{A-D-H} B. Ammann, M. Dahl, and E. Humbert, {\it Smooth Yamabe invariant and surgery},
 Journal of Differential Geometry  {\bf 94} (2013), 1:1-58.

\bibitem{Ammann-Dahl-Humbert} B. Ammann, M. Dahl, and E. Humbert, {\it The conformal Yamabe constant of product manifolds},
 Proceedings of the American Mathematical Society  {\bf 141} (2013), 295-307.

\bibitem{Aubin} T. Aubin,  {\it \'Equations diff\'erentielles non-lin\'eaires et probl\'eme de Yamabe concernant la courbure scalaire},
Journal de Math\'ematiques Pures et Appliqu\'ees {\bf 55} (1976), 3:269-296.

\bibitem{Bar-Dahl} C. B\"ar and M. Dahl, {\it Small eigenvalues of the conformal Laplacian}, Geometric and Functional Analysis, 
 {\bf 13} (2003), 483-508.

\bibitem{Bray-Neves} H. L. Bray and A. Neves, {\it  Classification of prime 3-manifolds with Yamabe invariant greater than $RP^3$}, Annals of Mathematics 
 {\bf 159} (2004), 407-424.


\bibitem{Djadli-Jourdain} Z. Djadli and A.  Jourdain, {\it Nodal solutions for scalar curvature type equations with perturbations terms on compact Riemannian manifolds},
   Bollettino della Unione Matematica Italiana. Serie VIII. Sez. B. Art.  Ric. Mat. {\bf 5} (2002),  1:205-226.




\bibitem{ElSayed} S. El Sayed, {\it Second eigenvalue of the {Y}amabe operator and applications},
     Calculus of Variations and Partial Differential Equations, {\bf 50} (2014),  3-4:665-692.


\bibitem{Hebey} E. Hebey, {\it Nonlinear Analysis on Manifolds: Sobolev Spaces and Inequalities}, Courant Lecture Notes  {\bf 5},  AMS/CIMS, New York (2000),  second edition, 290 pages. 

\bibitem{Hebey-Vaugon} E. Hebey and M. Vaugon, {\it Existence and multiplicity of nodal solutions for nonlinear elliptic equations with critical Sobolev growth},
     Journal of Functional Analysis, {\bf 119} (1994),  2:298-318.

\bibitem{Holcman} D. Holcman, {\it Solutions nodales sur les vari\'et\'es Riemannienes},
     Journal of Functional Analysis, {\bf 161} (1999),  1:219-245.

\bibitem{Jourdain} A.  Jourdain, {\it Solutions nodales pour des \'equations du type courbure scalaire sur la sph\'ere},
   Bulletin des Sciences Mathématiques, {\bf 123} (1999),  299-327.


\bibitem{K-W} J. Kazdan and F. W. Warner,  {\it Scalar curvature and conformal deformation of Riemannian
              structure}, Journal  of Differential Geometry {\bf 10} (1975), 113-134.

\bibitem{Kobayashi} O. Kobayashi,  {\it Scalar curvature of a metric with unit volume}, Mathematische Annalen,  {\bf 279} (1987), 2:253-265.

\bibitem{LeBrun} C. LeBrun,  {\it  Yamabe constant  and perturbed Seiberg-Witten equations}, Communications in Analysis and Geometry
 {\bf 5} (1997), 535-553.

    

\bibitem{Petean} J. Petean,  {\it Isoperimetric Regions in spherical cones and Yamabe constants of $M\times S^1$}, Geometriae Dedicata {\bf 143 } (2009), 37-48.

\bibitem{Petean2} J. Petean, {\it On nodal solutions of the {Y}amabe equation on products}, Journal of Geometry and Physics, {\bf 59} (2009), 10:1395-1401. 
    
\bibitem{Petean-Ruiz} J.  Petean and J. M.  Ruiz,  {\it Isoperimetric profile comparisons and Yamabe constants}, Annals of Global Analysis and Geometry,  {\bf 40} (2011), 177-189.

\bibitem{Petean-Ruiz2} J.  Petean and J. M.  Ruiz,  {\it On the  Yamabe constants of $S^2\times\re^3$ and $S^3\times\re^2$},  Differential Geometry and its Aplications,  {\bf 31} (2013), 2:308-319.

\bibitem{Pollack} D.  Pollack,  {\it Nonuniqueness and high energy solutions for a conformally invariant scalar
equation}, Communications in  Analysis and  Geometry,   {\bf 1} (1993), 3:347-414.

\bibitem{Schoen1} R.  Schoen,  {\it Conformal deformation of a {R}iemannian metric to constant
              scalar curvature},    Journal of Differential Geometry {\bf 20}  (1984), 2:479-495. 


\bibitem{Schoen} R.  Schoen,  {\it Variational theory for the total scalar curvature functional for Riemannian metrics and related topics},  Lecture Notes in Mathematics  {\bf 1365}, Springer-Verlag, Berlin (1987), 120-154. 


\bibitem{SchoenYau} R.  Schoen and S.-T. Yau,  {\it Conformally flat manifolds, {K}leinian groups and  curvature},  Inventiones Mathematicae {\bf 92} (1988),  1:47-71.

\bibitem{Trudinger} N. S.   Trudinger,  {\it Remarks concerning the conformal deformation of {R}iemannian
              structures on compact manifolds},  Annali della Scuola Normale Superiore di Pisa {\bf 22} (1968),  265-274.. 



\bibitem{Y} H.  Yamabe,  {\it On a deformation of {R}iemannian structures on compact
              manifolds}, Osaka Mathematical  Journal {\bf12} (1960), 21-37.


\end{thebibliography}
\end{document}